\documentclass[reqno, 12pt]{amsart}

\pdfoutput=1

\usepackage{enumerate}
\usepackage{latexsym}
\usepackage[centertags]{amsmath}
\usepackage{amsfonts}
\usepackage{amssymb}
\usepackage{amsthm}
\usepackage{newlfont}
\usepackage{graphics}
\usepackage{color}
\usepackage{float}
\usepackage{diagbox}
\usepackage{hyperref}

\usepackage{longtable}
\usepackage{rotating}
\usepackage{multirow}

\usepackage{extarrows}

\usepackage[sort,compress,numbers]{natbib}

\usepackage[utf8]{inputenc}

\newtheorem{thm}{Theorem}[section]
\newtheorem{cor}[thm]{Corollary}
\newtheorem{lem}[thm]{Lemma}
\newtheorem{prop}[thm]{Proposition}

\theoremstyle{mydefinition}

\theoremstyle{myremark}

\newtheorem{exa}[thm]{Example}

\newtheorem{prob}[thm]{Open problem}
\allowdisplaybreaks[4]

\title{On quotients of numerical semigroups for almost arithmetic progressions}

\author{Feihu Liu$^{1}$}

\address{ $^{1}$School of Mathematical Sciences,  Capital Normal University,
 Beijing 100048,  PR China}
\email{$^1$\texttt{liufeihu7476@163.com}}
\date{December 11, 2023}
\begin{document}
\maketitle

\begin{abstract}
Let $\langle A\rangle$ be the numerical semigroup generated by relatively prime positive integers $\{a_1,a_2,...,a_n\}$. The quotient of $\langle A\rangle$ with respect to a positive integer $p$ is defined by $\frac{\langle A\rangle}{p}=\{x\in \mathbb{N} \mid px\in \langle A\rangle\}$. The quotient $\frac{\langle A\rangle}{p}$ is known to be a semigroup but
is hard to study. When $p$ is a positive divisor of $a_1$, we reduce the computation of the Ap\'ery set of $\frac{a_1}{p}$ in $\frac{\langle A\rangle}{p}$ to a simple minimization problem. This allow us to obtain closed formulas of the Frobenius number of the quotient for some special numerical semigroups.
These includes the cases when $\langle A\rangle$ is the almost arithmetic progressions, the almost arithmetic progressions with initial gaps, etc. In particular, we partially solve an open problem proposed by A. Adeniran et al.
\end{abstract}

\def\D{{\mathcal{D}}}

\noindent
\begin{small}
 \emph{Mathematic subject classification (2020)}: Primary 20M14; Secondary 11D07, 11B75.
\end{small}

\noindent
\begin{small}
\emph{Keywords}: Quotient of a numerical semigroup; Almost arithmetic progression; Frobenius number; Genus.
\end{small}

\section{Introduction}

Let $\mathbb{Z}$, $\mathbb{N}$ and $\mathbb{P}$ be the set of all integers, non-negative integers and positive integers, respectively. A \emph{numerical semigroup} is a nonempty subset $S$ of $\mathbb{N}$ that is closed under the addition, $0\in S$, and $\mathbb{N}\setminus S$ is finite.

For a positive integers sequence $A=(a_1,a_2,...,a_n)$, $n\geq2$, if $\gcd(A)=1$, then the set
$$\langle A\rangle=\left\{a_1x_1+a_2x_2+\cdots +a_nx_n\ \mid  x_i\in \mathbb{N}\right\}$$
is a numerical semigroup (see \cite{J.C.Rosales}). We say that $A$ is a \emph{system of generators} of $\langle A\rangle$. This naturally attracts many scholars to study the following two invariants: 1). \emph{The Frobenius number} is denoted by $F(A)$ that is the greatest integer not belonging to $\langle A\rangle$. 2). \emph{The genus} (or \emph{Sylvester number}) is denoted by $g(A)$ that is the cardinality of $\mathbb{N}\backslash \langle A\rangle$.

Frobenius proposed the problem of giving a formula for $F(A)$. Sylvester \cite{J. J. Sylvester1,J. J. Sylvester2} obtained
\begin{equation}\label{FNGenusSylvester}
F(a_1,a_2)=a_1a_2-a_1-a_2\ \ \text{and}\ \ g(a_1,a_2)=\frac{1}{2}(a_1-1)(a_2-1).
\end{equation}
For $n\geq 3$, F. Curtis \cite{F.Curtis} proved that the $F(A)$ can not be given by closed formulas of a certain type. But many special cases have been considered, see \cite{Ramrez Alfonsn}.

In \cite{Rosales03}, a numerical semigroup $S(a,b,c)$ is called \emph{proportional modular} if
$S(a,b,c)=\{x\in \mathbb{N} \mid ax \mod b \leq cx\}$ and $a,b,c \in \mathbb{P}$. Suppose $\langle A\rangle$ is a numerical semigroup and $p\in \mathbb{P}$. A numerical semigroup $\frac{\langle A\rangle}{p}$ is called the \emph{the quotient of $\langle A\rangle$ by $p$} if $\frac{\langle A\rangle}{p}=\{x\in \mathbb{N} \mid px\in \langle A\rangle\}$. It is easy to verify that $S(a,b,c)$ and $\frac{\langle A\rangle}{p}$ are numerical semigroups (also see \cite{J.C.Rosales}). We can also get $\langle A\rangle \subseteq\frac{\langle A\rangle}{p}$ and $\frac{\langle A\rangle}{p}=\mathbb{N}$ if and only if $p\in\langle A\rangle$. In \cite{AMRobles08}, it is shown that a numerical semigroup is proportional modular if and only if it is the quotient of $\langle A\rangle$, where $A=(a,a+1), a\in \mathbb{P}$.

Suppose $A=(a_1,a_2)$, $\gcd(A)=1$ and $p\in \mathbb{P}$. A natural problem is to obtain the formulas for $F\left(\frac{\langle A\rangle}{p}\right)$ and $g\left(\frac{\langle A\rangle}{p}\right)$. In particular, it is an open problem to get the formulas of these two invariants (see \cite{MDelgado13}). Recently, some related developments can be found in \cite{F.Strazzanti,A.Adeniran,E.Cabanillas}.

We are concerned in this paper with the case $p\mid a_1$, i.e., the $p$ is a positive divisor of $a$. For $A=(a_1,a_2)$ and $\gcd(A)=1$, we can get $\frac{\langle A\rangle}{p}=\langle \frac{a_1}{p},a_2\rangle$ (see \cite{Rosales05}). By Equation \eqref{FNGenusSylvester}, we have
\begin{align*}
F\left(\frac{\langle a_1,a_2\rangle}{p}\right)=\frac{a_1a_2}{p}-\frac{a_1}{p}-a_2\ \ \text{and}\ \ g\left(\frac{\langle a_1,a_2\rangle}{p}\right)=\frac{1}{2p}(a_1-p)(a_2-1).
\end{align*}
The motivation of this paper comes from the paper \cite{A.Adeniran} that A. Adeniran et al. obtained the formulas of Frobenius number $F\left(\frac{\langle A\rangle}{p}\right)$ and the genus $F\left(\frac{\langle A\rangle}{p}\right)$ for $A=(a,a+d,a+2d)$, $A=(a,a+d,...,a+(a-1)d)$ under $\gcd(a,d)=1$ and $p\mid a$. They also proposed the following open problem:
\begin{prob}[\cite{A.Adeniran}]\label{OpenProb}
Let $A=(a,a+d,a+2d,...,a+kd)$, where $\gcd(a,d)=1$, $a,d\in\mathbb{P}$ and $3\leq k\leq a-2$. Find $F\left(\frac{\langle A\rangle}{p}\right)$ and $g\left(\frac{\langle A\rangle}{p}\right)$, as well as relationships between these quantities, for positive integers $p$.
\end{prob}

In this paper, we will partially solve the open problem, i.e, the case $p$ is a positive divisor of $a$. Furthermore, we will study the Frobenius formulas $F\left(\frac{\langle A\rangle}{p}\right)$ and $g\left(\frac{\langle A\rangle}{p}\right)$ for the \emph{almost arithmetic progressions} $A=(a,ha+d,ha+2d,...,ha+kd)$, the \emph{almost arithmetic progressions with initial gaps} $A=(a, ha+(K+1)d, ha+(K+2)d, ..., ha+kd)$, the \emph{almost arithmetic progressions with odd terms} $A=(a, ha+d, ha+3d, ha+5d, ..., ha+(2k+1)d)$ and $A=(a,ha-d,ha+d)$.

In this work we will modify the reduction method proposed in \cite{Fliuxin23} by authors. As a well-known result of Lemma  \ref{LiuXin001}, once a Ap\'ery set (see Section 2) of numerical semigroup is obtained, the Frobenius number and genus are easy to calculate. We reduce the computation of the elements in the Ap\'ery set of the quotient of a numerical semigroup by the positive integer $p$ to a simple minimization problem $O_B(M)$. By calculating $O_B(M)$, we can obtain an explicit formulas for the elements in the Ap\'ery set. The first main result obtained by this method is as follows:
\begin{thm}\label{FNgenusD}
Let $A=(a, dB)=(a, db_1, db_2, ..., db_k)$, $d,p\in \mathbb{P}$, $p\mid a$ and $\gcd(A)=1$. Then:
\begin{align*}
&F\left(\frac{\langle a,db_1,db_2,...,db_k\rangle}{p}\right)=d\cdot F\left(\frac{\langle a,b_1,b_2,...,b_k\rangle}{p}\right)+\frac{a(d-1)}{p},
\\&g\left(\frac{\langle a,db_1,db_2,...,db_k\rangle}{p}\right)=d\cdot g\left(\frac{\langle a,b_1,b_2,...,b_k\rangle}{p}\right)+\frac{(a-p)(d-1)}{2p}.
\end{align*}
\end{thm}

Next, we obtain the formulas $F\left(\frac{\langle A\rangle}{p}\right),g\left(\frac{\langle A\rangle}{p}\right)$ of the almost arithmetic progressions $A$.
\begin{thm}\label{Arith-ahd-p}
Let $A=(a, ha+d, ha+2d, ..., ha+kd)$, $\gcd(A)=1$,  $a,h,d,k,p\in \mathbb{P}$ and $1\leq k\leq a-1$. If $p\mid a$, then
\begin{align*}
F\left(\frac{\langle A\rangle}{p}\right)&=\left\lceil \frac{a-p}{k}\right\rceil\cdot \frac{ha}{p}+\frac{a(d-1)}{p}-d,
\\g\left(\frac{\langle A\rangle}{p}\right)&=h\cdot \sum_{r=1}^{\frac{a}{p}-1}\left\lceil\frac{rp}{k}\right\rceil+\frac{(a-p)(d-1)}{2p}.
\end{align*}
\end{thm}
When $h=1$, this is a partial answer to Open Problem \ref{OpenProb}. The other main results of this paper are as follows:
\begin{thm}\label{FNGenus-Kkdp}
Let $A=(a, ha+(K+1)d, ha+(K+2)d, ..., ha+kd)$, where $\gcd(A)=1$, $a, h, d, K, k, p\in \mathbb{P}$, $K\leq \frac{k-1}{2}$ and $a\geq 2$. If $p\mid a$, then
$$\begin{aligned}
F\left(\frac{\langle A\rangle}{p}\right)&=\left\{
\begin{aligned}
&\left\lceil \frac{a-p}{k}\right\rceil \cdot\frac{ha}{p}+\frac{a(d-1)}{p}-d\ & \text{if} &\ \ p>K, \\
&\left\lceil \frac{a+\left\lfloor\frac{K}{p}\right\rfloor\cdot p}{k}\right\rceil\cdot \frac{ha}{p}+\frac{a(d-1)}{p}+\left\lfloor\frac{K}{p}\right\rfloor\cdot d\ & \text{if} &\ \ p\leq K. \\
\end{aligned}
\right.
\\g\left(\frac{\langle A\rangle}{p}\right)&=\left\{
\begin{aligned}
&h\cdot \sum_{r=1}^{\frac{a}{p}-1}\left\lceil\frac{rp}{k}\right\rceil+\frac{(a-p)(d-1)}{2p}\ & \text{if} &\ \ p>K, \\
&h\cdot \left(\sum_{r=1}^{\left\lfloor\frac{K}{p}\right\rfloor}\left\lceil \frac{a+rp}{k}\right\rceil+ \sum_{r=\left\lfloor\frac{K}{p}\right\rfloor+1}^{\frac{a}{p}-1}\left\lceil\frac{rp}{k}\right\rceil\right)
+\frac{(a-p)(d-1)}{2p}+\left\lfloor\frac{K}{p}\right\rfloor\cdot d\ & \text{if} &\ \ p\leq K. \\
\end{aligned}
\right.
\end{aligned}$$
\end{thm}

\begin{thm}\label{wehaha}
Let $A=(a, ha-d, ha+d)$,  $d, h, p\in \mathbb{P}$, $\gcd(a, d)=1$, $\gcd(A)=1$, $ha-d>1$ and $s=\left\lfloor \frac{ha-d}{2hp}\right\rfloor$. If $p\mid a$, then
\begin{align*}
F\left(\frac{\langle A\rangle}{p}\right)&=\max\left\{\left\lfloor\frac{ha-d}{2hp}\right\rfloor(ha+d)-\frac{a}{p}, \left(\frac{a}{p}-\left\lceil \frac{ha-d}{2hp}\right\rceil\right)(ha-d)-\frac{a}{p}\right\},\\
g\left(\frac{\langle A\rangle}{p}\right)&=\frac{(ha+d)ps(s+1)}{2a}+\frac{(ha-d)(a-sp)
(a-sp-p)}{2pa}-\frac{a-p}{2p}.
\end{align*}
\end{thm}

\begin{thm}\label{FNgenus135p}
Let $A=(a, ha+d, ha+3d, ha+5d, ..., ha+(2k+1)d)$,  where $\gcd(A)=1$,  $a, h, d, k, p\in \mathbb{P}$, $a>2$, $3\leq 2k+1\leq a-1$. And let $a-p=(2k+1)s^{\prime}+t^{\prime}$,  where $1\leq t^{\prime}\leq 2k+1$.
If $p\mid a$ and $t^{\prime}\equiv 0\mod2$,  then we have
$$F\left(\frac{\langle A\rangle}{p}\right)=\frac{1}{p}\left(ha\left(\left\lfloor \frac{a-p-1}{2k+1}\right\rfloor +2\right)+(a-p)d-a\right).$$
\end{thm}
When $t^{\prime}$ is odd, the case is complicated. It is difficult to obtain a closed formula.

This paper is organized as follows.
In Section 2, we will provide a detailed introduction to our method about a reduction of Ap\'ery set of the numerical semigroup $\frac{\langle A\rangle}{p}$. We will give a proof of Theorem \ref{FNgenusD}.
In Section 3, we will provide the proofs of Theorem \ref{Arith-ahd-p}, \ref{FNGenus-Kkdp}, \ref{wehaha} and \ref{FNgenus135p}.

\section{A Reduction of Ap\'ery Set}

Suppose $A=(a_1,a_2,...,a_n)$, $a_i\in \mathbb{P}, 1\leq i\leq n$ and $\gcd(A)=1$. We know that the set
$$\langle A\rangle=\left\{a_1x_1+a_2x_2+\cdots +a_nx_n\ \mid  x_i\in \mathbb{N}\right\}$$
is a numerical semigroup. Let $a \in \langle A\rangle\backslash \{0\}$. The \emph{Ap\'ery set} of $a$ in $\langle A\rangle$ is $Ap(A,a)=\{w\in \langle A\rangle \mid w-a\notin \langle A\rangle\}$. In \cite{J.C.Rosales}, we can obtain
$$Ap(A,a)=\{N_0,N_1,N_2,...,N_{a-1}\},$$
where $N_r=\min\{ a_0\mid a_0\equiv r\mod a, \ a_0\in \langle A\rangle\}$, $0\leq r\leq a-1$. We usually take $a:=a_1$.

The following results were obtained by A. Brauer and J. E. Shockley \cite{J. E. Shockley} and E. S. Selmer \cite{E. S. Selmer} respectively. Then the Frobenius and genus are easy to compute.
\begin{lem}[\cite{J. E. Shockley}, \cite{E. S. Selmer}]\label{LiuXin001}
Suppose $A=(a_1,a_2,...,a_n)$ and $\gcd(A)=1$. The \emph{Ap\'ery set} of $a$ in $\langle A\rangle$ is $Ap(A,a)=\{N_0,N_1,N_2,...,N_{a-1}\}$. Then the Frobenius number and genus of $A$ are respectively:
\begin{align*}
F(A)=\max_{r\in \lbrace 0, 1, ..., a-1\rbrace}\{N_r\} -a,\ \ \ \ \ \
g(A)=\frac{1}{a}\sum_{r=1}^{a-1}N_r-\frac{a-1}{2}.
\end{align*}
\end{lem}

Note that $N_0=0$ is always 0. Suppose $d\in \mathbb{P}$ and $\gcd(a,d)=1$. If $r$ takes one of the remaining classes of $a$, then $dr$ is also. Therefore we have
$$\{N_0, N_1, N_2,..., N_{a-1}\}=\{N_{d\cdot 0}, N_{d\cdot 1}, N_{d\cdot 2},..., N_{d\cdot (a-1)}\}.$$

Suppose $A=(a, ha+dB)=(a, ha+db_1, ..., ha+db_k)$, $\gcd(A)=1$. We have $\gcd(a,d)=1$. Next we study the Ap\'ery set of the quotient of a numerical semigroup by a positive integer $p$, where $p$ is a positive divisor of $a$. The Ap\'ery set can be reduced to a simple minimization problem. This reduction will facilitate the calculation of two invariants: The Frobenius number $F\left(\frac{\langle A\rangle}{p}\right)$ and the genus $g\left(\frac{\langle A\rangle}{p}\right)$. When $p$ is a divisor of $a$, some results about the Ap\'ery set also see
[Chapter 5,\cite{J.C.Rosales}].

\begin{prop}\label{0202}
Suppose $A=(a, ha+db_1, ..., ha+db_k)$,  $k, h, d\in\mathbb{P}$ and $\gcd(A)=1$, $m\in\mathbb{N}$. Let $p$ be a positive divisor of $a$, i.e., $p\mid a$. The elements of Ap\'ery set of $\frac{a}{p}$ in the numerical semigroup $\frac{\langle A\rangle}{p}$ are
\begin{equation}\label{0203}
N_{dr,p}=\min \left\{O_B(ma+rp) \cdot \frac{ha}{p}+\left(\frac{ma}{p}+r\right)d \mid m\in \mathbb{N}\right\},
\end{equation}
where $0\leq r\leq \frac{a}{p}-1$ and
$$O_B(M)=\min\left\{\sum_{i=1}^kx_i \mid \sum_{i=1}^k b_ix_i=M, \ M,x_i\in\mathbb{N}, 1\leq i\leq k\right\}.$$
\end{prop}
\begin{proof}
By $\gcd(A)=1$, we have $\gcd(a,d)=1$. By the definition of the Ap\'ery set, for a given $0\leq r\leq \frac{a}{p}-1$, we have
\begin{align*}
N_{dr,p}&=\min\left\{ a_0\mid a_0\equiv dr\mod \frac{a}{p};\ a_0\in \frac{\langle A\rangle}{p}\right\}
\\&=\min\left\{ a_0\mid a_0\equiv dr \mod \frac{a}{p};\ pa_0\in \langle A\rangle\right\}
\\&=\min\left\{\sum_{i=1}^k\frac{(ha+db_i)}{p}x_i\mid \sum_{i=1}^k\frac{(ha+db_i)}{p}x_i\equiv dr\mod \frac{a}{p}, \ x_i\in\mathbb{N}, 1\leq i\leq k \right\}
\\&=\min\left\{\left(\sum_{i=1}^kx_i\right)\frac{ha}{p}+\sum_{i=1}^k\frac{x_ib_id}{p} \mid \sum_{i=1}^k(ha+db_i)x_i\equiv drp \mod a, \ x_i\in\mathbb{N}, 1\leq i\leq k\right\}
\\&=\min\left\{\left(\sum_{i=1}^kx_i\right)\frac{ha}{p}+\sum_{i=1}^k\frac{x_ib_id}{p} \mid \sum_{i=1}^kdb_ix_i\equiv drp \mod a, \ x_i\in\mathbb{N}, 1\leq i\leq k\right\}
\\&=\min\left\{\left(\sum_{i=1}^kx_i\right)\frac{ha}{p}+\sum_{i=1}^k\frac{x_ib_id}{p} \mid \sum_{i=1}^kb_ix_i\equiv rp \mod a, \ x_i\in\mathbb{N}, 1\leq i\leq k\right\}
\\&=\min\left\{\left(\sum_{i=1}^kx_i\right)\frac{ha}{p}+\left(\frac{ma}{p}+r\right)d \mid \sum_{i=1}^kb_ix_i=ma+rp, \ m,x_i\in\mathbb{N}, 1\leq i\leq k\right\}.
\end{align*}
Now for fixed $m$, and hence fixed $M=ma+rp$,
$\sum_{i=1}^kx_i$ is minimized to $O_B(ma+rp)$. This completes the proof.
\end{proof}

From the proof in Proposition \ref{0202} for $h=0$, we can obtain the proof of Theorem \ref{FNgenusD}.
\begin{proof}[Proof of Theorem \ref{FNgenusD}]
For a given $0\leq r\leq \frac{a}{p}-1$, we denote $N_{dr,p}(a,dB)$ and $N_{r,p}(a,B)$ as the element of the Ap\'ery set of $\frac{\langle a,db_1,db_2,...,db_k\rangle}{p}$ and $\frac{\langle a,b_1,b_2,...,b_k\rangle}{p}$ respectively. By the proof in Proposition \ref{0202}, we have
\begin{align*}
N_{dr,p}(a,dB)&=\min \left\{\sum_{i=1}^k\frac{x_ib_id}{p}\mid \sum_{i=1}^kx_ib_id\equiv drp \mod a,  x_i\in\mathbb{N}, 1\leq i\leq k \right\}
\\&=d\cdot \min \left\{\sum_{i=1}^k\frac{x_ib_i}{p}\mid \sum_{i=1}^kx_ib_i\equiv rp \mod a, x_i\in\mathbb{N}, 1\leq i\leq k\right\}
\\&=d\cdot N_{r,p}(a,B).
\end{align*}
The second ``=" in the above equation is because $\gcd(A)=1$, $\gcd(a,d)=1$. By Lemma \ref{LiuXin001}, we have
\begin{align*}
F\left(\frac{\langle a,db_1,db_2,...,db_k\rangle}{p}\right)+\frac{a}{p}&=\max_{0\leq r\leq \frac{a}{p}-1}N_{dr,p}(a,dB)=d\cdot \max_{0\leq r\leq \frac{a}{p}-1}N_{r,p}(a,B)
\\&=d\cdot \left(F\left(\frac{\langle a,b_1,b_2,...,b_k\rangle}{p}\right)+\frac{a}{p}\right).
\end{align*}
The first part theorem then follows. For the genus, by Lemma \ref{LiuXin001}, we have
\begin{align*}
g\left(\frac{\langle a,db_1,db_2,...,db_k\rangle}{p}\right)
&=\frac{p}{a}\sum_{r=1}^{\frac{a}{p}-1}N_{dr,p}(a,dB)-\frac{a-p}{2p}
=\frac{dp}{a}\sum_{r=1}^{\frac{a}{p}-1}N_{r,p}(a,B)-\frac{a-p}{2p}
\\&=d\left(g\left(\frac{\langle a,b_1,b_2,...,b_k\rangle}{p}\right)+\frac{a-p}{2p}\right)-\frac{a-p}{2p}.
\end{align*}
This completes the proof.
\end{proof}

If $p=1$ in Theorem \ref{FNgenusD}, then we can clearly get the following corollary.
\begin{cor}[\cite{J. E. Shockley},\cite{Rodseth78}]\label{CorolFNGrnus}
Let $A=(a, dB)=(a, db_1, db_2, ..., db_k)$, $d\in \mathbb{P}$ and $\gcd(A)=1$. Then we have
\begin{align*}
&F\left( a,db_1,db_2,...,db_k\right)=d\cdot F\left( a,b_1,b_2,...,b_k\right)+a(d-1),
\\&g\left( a,db_1,db_2,...,db_k\right)=d\cdot g\left(a,b_1,b_2,...,b_k\right)+\frac{(a-1)(d-1)}{2}.
\end{align*}
\end{cor}

We need to mention that the proof of Theorem \ref{FNgenusD} is quite simple. The proof of the Corollary \ref{CorolFNGrnus} was also collected in \cite{Ramrez Alfonsn}. But this proof process is relatively cumbersome. Moreover, we do not require $d=\gcd(db_1, db_2, ..., db_k)$.

For the convenience of the following discussion and by Proposition \ref{0202},  we can define an intermediate function with respect to $m$,  namely:
\begin{equation}\label{e-hd-11}
N_{dr,p}(m) :=O_B(ma+rp)\cdot \frac{ha}{p}+\left(\frac{ma}{p}+r\right)d,  \quad \text{where} \ A=(a, ha+dB)\ \text{and}\ \ p\mid a.
\end{equation}

Lemma \ref{0202} suggests the following strategy for $N_{dr,p}$ where
$A=(a, ha+dB)$: we first try to solve $O_B(M)$ for general $M$. If we have a formula that is nice enough,  then we can analyze $N_{dr,p}(m)$. In fact, if $N_{dr,p}(m)$ increases with $m$, then we have a formula for $N_{dr,p}=N_{dr,p}(0)$. Hence we can further obtain the formula of $F\left(\frac{\langle A\rangle}{p}\right)$ and $g\left(\frac{\langle A\rangle}{p}\right)$.

\section{The Proofs of Theorem \ref{Arith-ahd-p}, \ref{FNGenus-Kkdp}, \ref{wehaha} And \ref{FNgenus135p}}

In this section,  we will first give the proof of the formulas $F\left(\frac{\langle A\rangle}{p}\right),g\left(\frac{\langle A\rangle}{p}\right)$ of almost arithmetic sequences $A$.
The central problem is to find an explicit formula of $N_{dr,p}$.

\begin{proof}[Proof of Theorem \ref{Arith-ahd-p}]
By the Equation \eqref{e-hd-11}, we have
$$N_{dr,p}(m)=O_B(ma+rp)\cdot \frac{ha}{p}+\left(\frac{ma}{p}+r\right)d,$$
where $O_B(ma+rp)=\min\left\{\sum_{i=1}^kx_i \mid \sum_{i=1}^ki\cdot x_i=ma+rp, x_i\in \mathbb{N}\right\}$.

If $ma+rp=s\cdot k+t$ with $1\leq t\leq k$,  then $O_B(ma+rp)=s+1$,  which
minimizes at  $x_k=s, x_t=1, x_i=0, (i\neq k, t)$ when $t\neq k$,
and minimizes at $x_k=s+1, x_i=0, (i\neq k)$ when $t=k$. Therefore
$$N_{dr,p}(m)=\frac{(s+1)ha}{p}+\left(\frac{ma}{p}+r\right)d=\frac{ha}{p}\left\lceil  \frac{ma+rp}{k}\right\rceil +\left(\frac{ma}{p}+r\right)d$$
is increasing with respect to $m$ for a given $r$. It follows that
$$N_{dr,p}=N_{dr,p}(0)=\frac{ha}{p}\left\lceil  \frac{rp}{k}\right\rceil +rd.$$
This ia a explicit formula respect to $r$. Moreover the $N_{dr}$ is increasing with respect to $r$. By Lemma \ref{LiuXin001}, we have the Frobenius number
\begin{align*}
F\left(\frac{\langle A\rangle}{p}\right)=\mathop{\max}\limits_{r\in \{0, 1, ...,\frac{a}{p}-1\}} \{ N_{dr,p}\}-\frac{a}{p}=N_{d(\frac{a}{p}-1),p}-\frac{a}{p}=\frac{ha}{p}\left\lceil  \frac{a-p}{k}\right\rceil+\left(\frac{a}{p}-1\right)d-\frac{a}{p}.
\end{align*}
Similarly, we have the genus
\begin{align*}
g\left(\frac{\langle A\rangle}{p}\right)=\frac{p}{a}\cdot \sum_{r=1}^{\frac{a}{p}-1}N_{dr,p}-\frac{a-p}{2p}
=h\cdot \sum_{r=1}^{\frac{a}{p}-1}\left\lceil\frac{rp}{k}\right\rceil+\frac{(a-p)(d-1)}{2p}.
\end{align*}
This completes the proof.
\end{proof}

\begin{exa}
In Theorem \ref{Arith-ahd-p}, let $a=84,h=3,d=101,k=4$. Now, we have $A=(84,353,454,555,656)$. If $p=14$, then $F\left(\frac{\langle A\rangle}{21}\right)=823$ and $g\left(\frac{\langle A\rangle}{21}\right)=412$. If $p=21$, then $F\left(\frac{\langle A\rangle}{21}\right)=491$ and $g\left(\frac{\langle A\rangle}{21}\right)=249$.
\end{exa}

We can use the  \emph{numericalsgps GAP} package (\cite{M.Delgado}) verify the correctness for the Frobenius number and genus of above the quotient of numerical semigroup.

If $p=1$ in Theorem \ref{Arith-ahd-p}, then Frobenius number $F(A)$ and genus $g(A)$ are obtained by E. S. Selmer in \cite{E. S. Selmer}. If further $h=1$, then J. B. Roberts get $F(A)$ and $g(A)$ in \cite{Roberts1}. For $A=(a,a+d,...,a+kd)$, the $F(A)$ and $g(A)$ were first obtained by A. Brauer in \cite{A. Brauer}.

Next we consider the proof of Theorem \ref{FNGenus-Kkdp} about the almost arithmetic progressions with initial gaps.
\begin{proof}[Proof of Theorem \ref{FNGenus-Kkdp}]
By the Equation \eqref{e-hd-11}, we have
$$N_{dr,p}(m)=O_B(ma+rp)\cdot \frac{ha}{p}+\left(\frac{ma}{p}+r\right)d,$$
where $O_B(ma+rp)=\min\left\{\sum_{i=1}^{k-K}x_i \mid \sum_{i=1}^{k-K}(K+i)\cdot x_i=ma+rp, x_i\in \mathbb{N} \right\}$.

Let $ma+rp=k\cdot s+t$,  where $s\geq 0, 1\leq t\leq k$.
When $K+1\leq t\leq k$,  we have $O_B(ma+rp)=s+1$,  which minimizes at $x_{k-K}=s,  x_{t-K}=1,  x_i=0, (i\neq k-K, t-K), t\neq k$ and $x_{k-K}=s+1, x_i=0, (i\neq k-K), t=k$.
When $1\leq t\leq K$, there is no solution for $s=0$. For $s\geq 1$,  we have $ma+rp=(s-1)k+k+t$ and $O_B(M)=s+1$,  which minimizes at $x_{k-K}=s-1,  x_{k+t-2K-1}=1,  x_1=1,  x_i=0,  (i\neq 1, k+t-2K-1, k-K)$; This is because $K\leq \frac{k-1}{2}$.

Therefore, we have
$$N_{dr,p}(m)=\frac{(s+1)ha}{p}+d\left(\frac{ma}{p}+r\right)=\left\lceil \frac{ma+rp}{k}\right\rceil\cdot \frac{ha}{p}+d\left(\frac{ma}{p}+r\right),$$
which is increasing with respect to $m$. Therefore $N_{dr,p}=N_{dr,p}(0)$ for $s\neq 0$ or $K+1\leq t\leq k$ and $N_{dr,p}=N_{dr,p}(1)$ for $s=0, 1\leq t\leq K$
(note that $N_{dr,p}(0)$ does not exit in this case).

Thus we have
$$\begin{aligned}
N_{dr,p}=
\left\{
\begin{aligned}
&\frac{ha}{p}\left\lceil \frac{a+rp}{k}\right\rceil+d\left(\frac{a}{p}+r\right)\ & \text{if} & \ \ s=0, 1\leq t\leq K, \\
&\frac{ha}{p}\left\lceil \frac{rp}{k}\right\rceil+dr\ & \text{if} & \ \ \text{otherwise}. \\
\end{aligned}
\right.
\end{aligned}$$
For a given $r$, if $p>K$, then we have $N_{dr,p}=N_{dr,p}(0)=\frac{ha}{p}\lceil \frac{rp}{k}\rceil+rd$. This is increasing with respect to $r$. By Lemma \ref{LiuXin001}, we have the Frobenius number
$$F\left(\frac{\langle A\rangle}{p}\right)=N_{d\left(\frac{a}{p}-1\right),p}-\frac{a}{p}=\left\lceil \frac{a-p}{k}\right\rceil \cdot\frac{ha}{p}+\left(\frac{a}{p}-1\right)d-\frac{a}{p},$$
and the genus
$$g\left(\frac{\langle A\rangle}{p}\right)=\frac{p}{a}\sum_{r=1}^{\frac{a}{p}-1}N_{dr,p}-\frac{a-p}{2p}=h\cdot \sum_{r=1}^{\frac{a}{p}-1}\left\lceil\frac{rp}{k}\right\rceil+\frac{(a-p)(d-1)}{2p}.$$
If $p\leq K$, we assume $K=cp+c^{\prime}$, $0\leq c^{\prime}\leq p-1$. So we have $c=\left\lfloor\frac{K}{p}\right\rfloor\geq 1$. Now, the $N_{dr,p}$ is reduced to
$$\begin{aligned}
N_{dr,p}=
\left\{
\begin{aligned}
&\frac{ha}{p}\left\lceil \frac{a+rp}{k}\right\rceil+d\left(\frac{a}{p}+r\right)\ & \text{if} & \ \ 1\leq r\leq c, \\
&\frac{ha}{p}\left\lceil \frac{rp}{k}\right\rceil+dr\ & \text{if} & \ \ c+1\leq r\leq \frac{a}{p}-1. \\
\end{aligned}
\right.
\end{aligned}$$
We can obviously get
$$\mathop{\max}\limits_{r\in \{0, 1, ...,\frac{a}{p}-1\}}\{N_{dr,p}\}=\max\{N_{dc,p},N_{d(\frac{a}{p}-1),p}\}=N_{dc,p}.$$
By Lemma \ref{LiuXin001}, we have the Frobenius number
$$F\left(\frac{\langle A\rangle}{p}\right)=N_{dc,p}-\frac{a}{p}=\left\lceil \frac{a+cp}{k}\right\rceil \cdot\frac{ha}{p}+\left(\frac{a}{p}+c\right)d-\frac{a}{p},$$
and the genus
$$g\left(\frac{\langle A\rangle}{p}\right)=\frac{p}{a}\sum_{r=1}^{\frac{a}{p}-1}N_{dr,p}-\frac{a-p}{2p}=h\cdot \left(\sum_{r=1}^c \left\lceil\frac{a+rp}{k}\right\rceil +\sum_{r=c+1}^{\frac{a}{p}-1}\left\lceil\frac{rp}{k}\right\rceil\right)+\frac{(a-p)(d-1)}{2p}+dc.$$
This completes the proof by $c=\left\lfloor\frac{K}{p}\right\rfloor$.
\end{proof}

\begin{exa}
Suppose $a=86,h=5,d=9,K=2,k=6$ in Theorem \ref{FNGenus-Kkdp}. Then we have $A=(86,457,466,475,484)$. If $p=43>2=K$, then $F\left(\frac{\langle A\rangle}{43}\right)=87$ and $g\left(\frac{\langle A\rangle}{43}\right)=44$. Let $a=300,h=4,d=7,K=6,k=13$. We have $A=(300,1249,1256,1263,1270,1277,1284,1291)$.
If $p=5<6=K$, then $F\left(\frac{\langle A\rangle}{5}\right)=6127$ and $g\left(\frac{\langle A\rangle}{5}\right)=3108$.
\end{exa}

If $p=1$ in Theorem \ref{FNGenus-Kkdp}, the Frobenius number $F(A)$ and the genus $g(A)$ are obtained by authors in \cite{Fliuxin23}. Furthermore, if $p=h=1$, the $F(A)$ and $g(A)$ got by Takao Komatsu in \cite{T.Komatsu22Arx}.

\begin{proof}[Proof of Theorem \ref{wehaha}]
By Lemma \ref{0202}, we consider $O_B(M)=\min\{x_1+x_2 \mid x_2-x_1=ma+rp\}$ for $m\in \mathbb{Z}$. If $m\ge 0$,  then $x_1+x_2=2x_1+ma+rp$ minimizes to $ma+rp$ at $x_1=0$,  so that $N_{dr,p}$ minimizes to $T_1:=(ha+d)r$ at $m=0, \ x_1=0$.
If $m<0$ then $ma+rp<0$ and $x_1+x_2=2x_2-ma-rp$ minimizes to $-ma-rp$ at $x_2=0$,  so that $N_{dr,p}$ minimizes to $T_2:=\frac{(ha-d)(a-rp)}{p}$ at $x_2=0, \ m=-1$.

Solving $T_1-T_2\geq 0$ gives $r\geq s=\frac{ha-d}{2hp}$. Thus  $N_{dr}=T_2$
for $r\geq s=\frac{ha-d}{2hp}$,  and $N_{dr}=T_1$ for otherwise. Obviously $0< s< \frac{a}{p}$.    Therefore we have
$$\max_{1\leq r\leq \frac{a}{p}-1}\{N_{dr,p}\}=\max\left\{\left\lfloor\frac{ha-d}{2hp}\right\rfloor(ha+d),\left(\frac{a}{p}
-\left\lceil\frac{ha-d}{2hp}\right\rceil\right)(ha-d)\right\}.$$
By Lemma \ref{LiuXin001}, we get the formulas of Frobenius number $F\left(\frac{\langle A\rangle}{p}\right)$ and genus $g\left(\frac{\langle A\rangle}{p}\right)$.
\end{proof}

\begin{exa}
In Theorem \ref{wehaha}, let $a=1120,h=7,d=9$. We have $A=(1120,7831,7849)$. If $p=28$, then $F\left(\frac{\langle A\rangle}{28}\right)=156580$ and $g\left(\frac{\langle A\rangle}{28}\right)=78376$.
\end{exa}

\begin{proof}[Proof of Theorem \ref{FNgenus135p}]
We first consider
$$O_B(ma+rp)=\min\left\{\sum_{i=0}^kx_i \mid x_0+\sum_{i=1}^k(2i+1)\cdot x_i=ma+rp\right\}.$$
Suppose $ma+rp=(2k+1)\cdot s+t$,  where $1\leq t\leq 2k+1$.

If $t$ is even,  we have $O_B(ma+rp)=s+2$ and
$$N_{dr,p}(m)=\frac{ha(s+2)}{p}+d\left(\frac{ma}{p}+r\right)=\frac{ha}{p}\left(\left\lfloor \frac{ma+rp-1}{2k+1}\right\rfloor +2\right)+d\left(\frac{ma}{p}+r\right).$$
If $t$ is odd,  we have $O_B(M)=s+1$ and
$$N_{dr,p}(m)=\frac{ha(s+1)}{p}+d\left(\frac{ma}{p}+r\right)=\frac{ha}{p}\left(\left\lfloor \frac{ma+rp-1}{2k+1}\right\rfloor +1\right)+d\left(\frac{ma}{p}+r\right).$$
By $N_{dr,p}(m+1)-N_{dr,p}(m)\geq \frac{ha}{p}\left(\left\lfloor\frac{a}{2k+1}\right\rfloor-1\right)+ \frac{da}{p}\geq 0$,  we have
$$\begin{aligned}
N_{dr,p}=N_{dr,p}(0)=
\left\{
\begin{aligned}
&\frac{ha}{p}\left(\left\lfloor \frac{rp-1}{2k+1}\right\rfloor +2\right)+dr \ \ & \text{if}\ \ & t\ \text{is even}, \\
&\frac{ha}{p}\left(\left\lfloor \frac{rp-1}{2k+1}\right\rfloor +1\right)+dr \ \ & \text{if}\ \ & t\ \text{is odd}. \\
\end{aligned}
\right.
\end{aligned}$$
By $rp\leq a-p=(2k+1)s^{\prime}+t^{\prime}$ and Lemma \ref{LiuXin001}, if $t^{\prime}$ is even, then the Frobenius number is
$$F\left(\frac{\langle A\rangle}{p}\right)=N_{d\left(\frac{a}{p}-1\right),p}-\frac{a}{p}=\frac{1}{p}\left(ha\left(\left\lfloor \frac{a-p-1}{2k+1}\right\rfloor +2\right)+(a-p)d-a\right).$$
This completes the proof.
\end{proof}

\begin{exa}
In Theorem \ref{FNgenus135p}, let $a=33,h=4,d=5,k=2$. We have $A=(33,137,147,157)$. If $p=11$, then $F\left(\frac{\langle A\rangle}{11}\right)=79$.
\end{exa}

\section{Concluding Remark}

Let $d\in\mathbb{Z}\backslash \{0\}$. We require $ha+db_i>1$ for all $1\leq i\leq k$ if $d$ is a negative integer. Then Proposition \ref{0202} is still correct. However, the $F\left(\frac{\langle A\rangle}{p}\right)$ of almost arithmetic progressions is false. At this time, we need to consider some constraints in order to $N_{dr,p}(m+1)-N_{dr,p}(m)\geq 0$ and analysis $\max\{N_{dr,p}\}$ in the proof of Theorem \ref{Arith-ahd-p}.

\noindent
{\small \textbf{Acknowledgements:}
The authors would like to show their sincere appreciations to all suggestions improving the presentation of this paper. Particular thanks go to Professor Guoce. Xin who critically read the paper and made numerous helpful suggestions. This work was partially supported by the National Natural Science Foundation of China [12071311].

\end{document}